\documentclass[12pt]{article}
\usepackage{t1enc,amssymb,amsbsy,amsthm,amsmath,amsfonts}
\usepackage{graphicx}
\usepackage{color}

\newcommand{\RR}{\mathbb{R}}
\newcommand{\e}{\rm e}

\newtheorem{theorem}{Theorem}[section]
\newtheorem{lemma}[theorem]{Lemma}
\newtheorem{proposition}[theorem]{Proposition}
\newtheorem{corollary}[theorem]{Corollary}

\newtheorem{remark}[theorem]{Remark}

\numberwithin{equation}{section}

\def\ln{\hbox {\rm ln\,}}

\def\e{\hbox {\rm e}}

\def\P{{\bf P}}
\def\R{{\bf R}}
\def\R{{\bf R}}

\def\E{{\bf E}}

\def\cF{{\cal F}}
\def\cD{{\cal D}}

\def\cC{{\cal C}}

\def\cT{{\cal T}}

\def\rb{\rbrace}
\def\lb{\lbrace}
\def\rp{\right)}
\def\lp{\left(}

\def\al{\alpha}

\def\la{\lambda}

\begin{document}

\author{
Julien Randon-Furling
\\{\small Universit\'e Panth\'eon Sorbonne}
\\{\small SAMM - FP2M (CNRS FR2036)}
\\{\small F-75013 Paris, France}
\\{\small and}
\\{\small UM6P - MSDA, 43150 Ben Guerir, Maroc}
\\{\small email: j.randon-furling@cantab.net}
\and
Paavo Salminen\\{\small Abo Akademi University}
\\{\small Faculty of Science and Engineering}
\\{\small FIN-20500 Abo, Finland} \\{\small email: phsalmin@abo.fi}
\and
Pierre Vallois
\\{\small Universit\'e de Lorraine}
\\{\small CNRS, Inria, IECL}
\\{\small F-54000 Nancy, France}
\\{\small email: pierre.vallois@univ-lorraine.fr}
}

\title{On a first hit distribution of the running maximum of Brownian motion}
\date{}
\maketitle

\begin{abstract} Let $(S_t)_{t\geq 0}$ be the running maximum of a standard Brownian motion $(B_t)_{t\geq 0}$ and $T_m:=\inf\{t; \, mS_t<t\},\, m>0$. In this note we calculate the joint distribution of $T_m$ and $B_{T_m}$. The motivation for our work comes from a mathematical model for animal foraging. We also present results for   Brownian motion with drift. 
\end{abstract}

\eject

\medskip
\section{Introduction}
\label{sec0}

A part of the motivation behind the study presented here stems from a toy-model designed by Paul~Krapivsky for animal foraging~\cite{pk17}. Among many applications, stochastic processes have indeed often been used to model the paths traced by animals searching for food, shelter or other necessities~\cite{vis11}.
\begin{figure}[h!]
\label{fig1}
\centering
\includegraphics[scale=0.4]{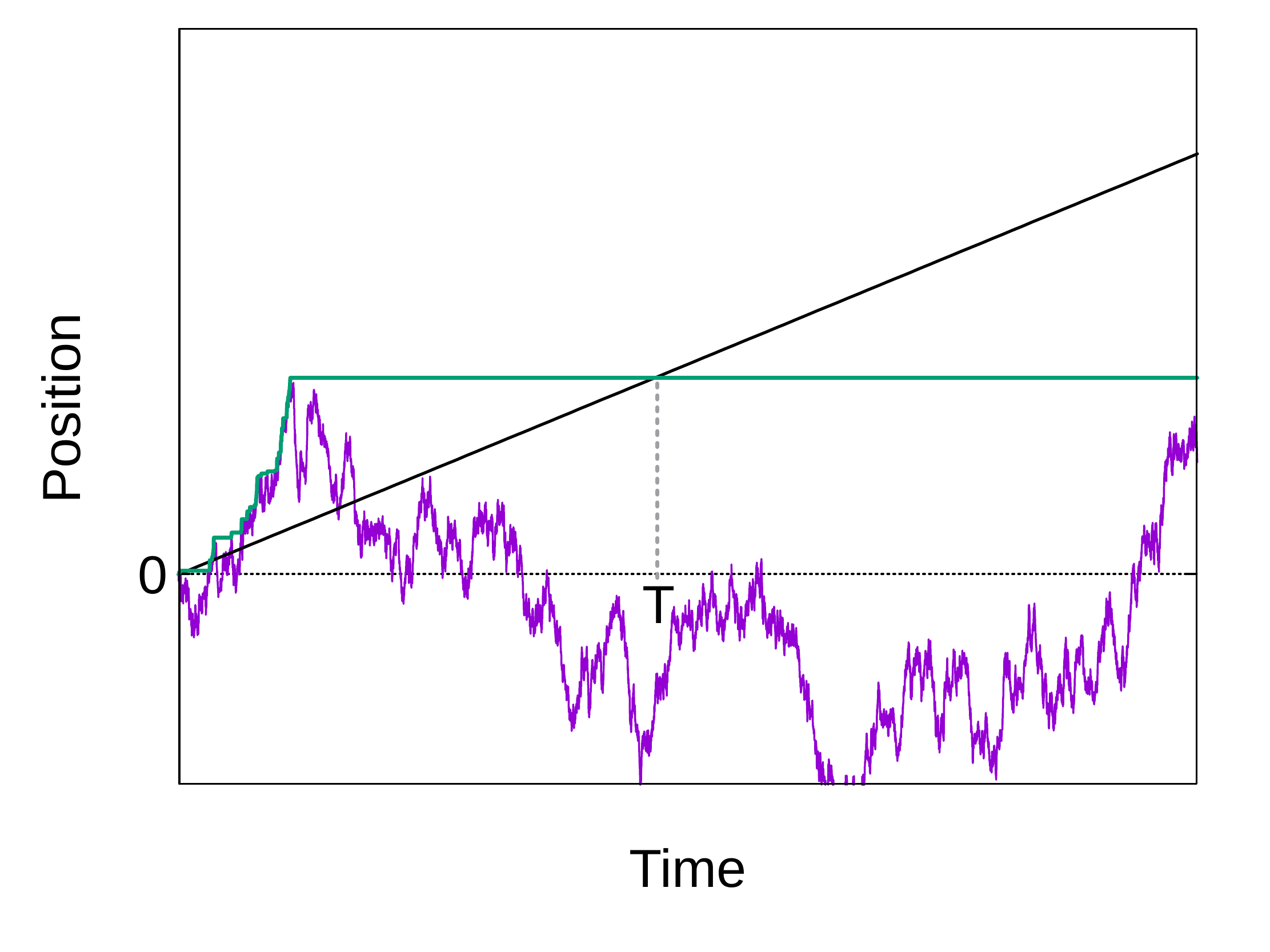}
\caption{Position of an animal foraging in a one-dimensional space, modelled as standard Brownian motion. Shown are also the supremum process of the Brownian motion and $T$, the first hitting time of the supremum on the diagonal barrier.}
\end{figure}

The toy-model which we have in mind here deals with the simplified, stylized case of an animal foraging in a one-dimensional space. The animal's initial position coincides with the origin, and we model its position as time~$t$ elapses by a standard Brownian motion $(B_t)_{t\geq 0}$. For the sake of simplicity, we suppose that the forager's metabolism is basic: to survive, it needs one unit of food per unit time, and it may stockpile any extra supply for future use, without any upper limit on the size of the stock nor any expiry date for the consumption thereof. As for the provision of food, we assume that only half of the space (say, the positive half-line) is initially filled with one unit of food per unit length, and that there is no replenishment. Thus, after a time $t$, the forager has absorbed an amount of food equal to $S_t$, its maximal displacement in the positive direction. For the forager to survive up to a time~$t$, it should be the case that, at every time $s \leq t$, the amount of food it had absorbed was not less than $s$. In other terms, the probability that the forager survives up to a time~$t$ is given by the probability that $S_s \geq s$ for all $s \in [0,t]$. Equivalently, this is the probability that the first (downward) hitting time $T$ of the supremum process $(S_t)_{t\geq 0}$ on the diagonal barrier occurs after $t$, as shown in~Figure~\ref{fig1}.

A natural extension of the original problem consists of the so-called double-sided case, that is, there is food on both sides for the forager/animal  to find. The survival probability at time $t$ then becomes the probability that the {range} $R_s :=\sup_{r\leq s} B_r - \inf_{r \leq s} B_k$ of a standard Brownian motion always remains greater than $s$ for all $s\leq t$. It is  an open question to determine  the distribution of the survival time in this case.

To study the distribution and properties of the supremum of a stochastic process is a very classical and central topic in the theory of stochastic processes. As well known, for Brownian motion the distribution of $S_t$ can be found using a path transformation, that is, D. Andr\'e's reflection principle. The process $(S_t)_{t\geq 0}$ can also be seen as a local time process of a reflecting Brownian motion due to the profound result by P. L\'evy characterizing the process   $(S_t-B_t)_{t\geq 0}$ as a reflecting Brownian motion. There are also a number of papers  devoted to  the joint law of the supremum, the position and the random timepoint when the supremum is attained, in particular, for diffusions.  In this occasion we wish to refer to a work by L. Shepp \cite{shepp79} where the distribution is found for a Brownian motion with drift. As explained in Section \ref{sec21} the distribution of $T$ has been calculated by R. Doney in \cite{doney91}. Our main contribution in this paper is to find the joint density of $T$ and $B_T$.


We consider first the case without drift in Section \ref{sec21}, and then apply Girsanov's theorem to find the distribution for Brownian motion with drift $\mu\not=0$ in Section \ref{sec22}. The proof for standard Brownian motion does not, however, explain how the distribution was originally found. The presented proof is a verification, that is, we charaterize the density as a unique solution of an integral equation and show that our candidate density solves the equation. We have  three approaches for calculating the candidate density. The first one is based on path transformations and the second  one on analysing the inverse of the running maximum process $(S_t)_{t\geq 0}$ combined with some  formulas from the Brownian excursion theory. The third approach is to study the problem in a discrete setting and anticipate a passage to limit to obtain the formula for standard Brownian motion.  Unfortunately, in all these approaches there are some technical difficulties which we have not been able to resolve up to now.  Because of this, we do not treat   these approaches in detail in this paper but hope to return to this issue in a forthcoming publication. However, some indications  concerning the L\'evy process approach are given in Remark \ref{LPA}, and the path transformation method is discussed in Section~\ref{sec31}.

\section{Joint  distribution of $T$ and $B_T$ for standard Brownian motion }
\label{sec21}
Let $B=(B_t)_{t\geq 0}$ be a standard Brownian motion initiated at $0$,
$$
S_t:=\sup\{B_s\,; \, 0\leq s\leq t\}
$$
its running supremum up to a fixed time $t>0$, and for $m>0$ 
\begin{equation}\label{defT}
  T_m:=\inf\{t\,;\, mS_t<t\}
\end{equation}
the first time when the process $(mS_t-t)_{t\geq 0}$ becomes (strictly) negative. Notice also that, by continuity, $S_{T_m}=T_m/m$. We let $\P_x$ and   $\E_x$ denote the probability measure and the expectation of a Brownian motion when initiated from an arbitrary point $x.$ In this section we find the joint $\P_0$-density of  $T_m$ and $B_{T_m}$. The focus  is first on the distribution of  $T_m$. We use  the theory of L\'evy processes from Doney \cite{doney91}, which  yields  the  Laplace transform of $T_m$, see ibid p. 572. To make the paper more self contained we give anyway the main points of the derivation.  Of course, the distribution of $T_m$ can also be obtained from the joint distribution of $T_m$ and $B_{T_m}$ presented in Theorem \ref{thm1}. 

\begin{remark}
\label{rmk2}
 {\it It is, in fact, enough to find the joint distribution of $T_m$ and $B_{T_m}$ "only" for $m=1$ and use the scaling property of Brownian motion to deduce the distribution for a general $m>0$. To see this, let 
 $\bar B_t=B_{m^2 t}/m $. Then $(\bar B_t)_{t\geq 0}$ is a Brownian motion  and
$$T_1(\bar B)=\inf\big\{t\geq 0,\; \bar S_t<t\big\},$$
where $\bar S_t=\sup_{0\leq u\leq t}\bar B_u$.  Now we have a.s.
\begin{eqnarray}
\label{scalprop0}
\nonumber 
  T_1(\bar  B)&=&\inf\big\{t\geq 0,\; \frac{1}{m} S_{m^2 t}<t\big\}=  \frac{1}{ m^2} \inf\big\{m^2 t, \; m S_{m^2 t}<m^2 t \big\}\\
   &=& \frac{1}{m^2}\,T_m,
\end{eqnarray}
and, further, a.s.
\begin{eqnarray*}
\bar  B_{ T_1(\bar  B)}&=&\frac{1}{m} B_{m^2T_1(\bar  B) }=\frac{1}{m} B_{T_m}.
\end{eqnarray*}
Consequently,
\begin{equation}
\label{scalprop} 
( {T}_m, B_{T_m})\overset{(d)}{=}( m^2{T_1(\bar B)}, m\bar B_{T_1 (\bar B)}).
\end{equation}
} 
\end{remark} 

\begin{remark} \label{que} Recall that $(S_t)_{t\geq 0}$ has the same law as
$(L_t)_{t\geq 0}$, where $L_t$ is the local time at 0 of a reflecting Brownian motion $(|B_t|)_{t\geq 0}$ defined via
$$
L_t:=\lim_{\varepsilon\downarrow 0}\frac 1{2\varepsilon}\int_0^t {\bf 1}_{[0,\varepsilon)}(|B_s|)\, ds\qquad a.s.
$$
The processes of the type $(L_t-t)_{t\geq 0}$ has been introduced and analyzed as models for fluid queues. For this , see, in particular, \cite{mannersalonorros}, where $(L_t)_{t\geq 0}$ is the local time at 0 of a reflecting Brownian motion with negative drift, and  \cite{konstantopoulosetal},  where a more general setting is considered and also  further references can be found. In these articles the main interest is in finding the distribution of the length of a busy period (and also of the idle period) under  the stationary probability measure associated with the underlying process.
\end{remark}

\begin{theorem}
\label{thm0}
The random time $T_m$ defined by \eqref{defT} is almost surely positive and finite. Its density is 
\begin{equation}
\label{e1}
\P_0(T_m\in dt)/dt=\frac 1m\sqrt{\frac 2{\pi t}}\lp\e^{-t/2m^2}-\frac 1{m}
\int_{t/m}^\infty\e^{-y^2/2t} dy\rp.
\end{equation}
Moreover, the Laplace transform of $T_m$ is for $\alpha>-1/(2m^2)$ given by
\begin{align}
\label{e2}
\E_0(\e^{-\alpha T_m})&=\frac 2{1+ \sqrt{1+2m^2\alpha}}\\
&
\label{e2a}
=\frac 1{\alpha m^2}\lp \sqrt{1+2\alpha m^2}-1\rp,
\end{align}
and, hence,  $T_m$  has all (positive) moments.
\end{theorem}

\begin{remark}
\label{rmk22}
{\it  Let $f_1$ denote the density of $T_1$. Then
\begin{equation}
\label{e3}
f_1(t)=2g(t)-1+G(t),
\end{equation}
where $g$ is the density of the gamma distribution with parameters $1/2$ and $1/2$ and $G$ is the corresponding distribution function. Using  (\ref{scalprop0}) we get 
\begin{equation}
\label{e33}
f_m(t)=\frac 1{m^2}\lp 2\,g\lp\frac t{m^2}\rp-1+G\lp\frac t{m^2}\rp\rp,
\end{equation} 
where $f_m$ denotes the density of $T_m.$ The identity (\ref{e33}) can also be checked directly from (\ref{e1}).
}
\end{remark}

\begin{corollary}
\label{cor00}
The distribution function of $T_m$ is given by
\begin{equation}
\label{e330}
\P_0\lp T_m>t\rp=1-G\lp\frac t{m^2}\rp-\frac t{m^2}f_1\lp\frac t{m^2}\rp,
\end{equation}
where $f_1$ and $G$ are as defined in the Remark \ref{rmk22}. The moments of $T_m$ are given for $k=1,2,\dots$ by
\begin{equation}
\label{e3300}
\E_0\lp T^k_m\rp=\frac {1\cdot3\cdot ...\cdot (2k-1)}{(k+1)!}\, {m^{2k}}
\end{equation}

\end{corollary}

\noindent
{\it Proof of Theorem \ref{thm0}.} Due to scaling, as explained in Remark \ref{rmk2}, we assume without loss of generality that $m=1$, and introduce $T:=T_1$.   Let $(\cT_t)_{t\geq 0}$ denote the right continuous inverse of $(S_t)_{t\geq 0},$ i.e.
$$
\cT_t:=\inf\{u\,;\, S_u>t\}.
$$
It is well known that $(\cT_t)_{t\geq 0}$ is a $1/2$-stable subordinator. Hence, the process $X=(X_t)_{t\geq 0}$ defined  by
$$
X_t:=t-\cT_t 
$$
is a spectrally negative L\'evy process of bounded variation having the Laplace exponent 
\begin{equation}
\label{ee100}
\E_0\lp\e^{\lambda X_t}\rp=\E_0\lp\e^{\lambda(t-\cT_t}\rp=\e^{t(\lambda-\sqrt{2\lambda})}, \ \lambda\geq 0.
\end{equation}
The key observation is that
\begin{align} \label{defH_0}
 & H_0:=   \inf\{t \,;\, X_t<0\}
 = \inf\{t \,;\, \cT_t>t\}= \inf\{u  \,;\, S_u<u \}=T.
 \end{align} 
From the theory of L\'evy processes we know that the process $X$ satisfies  
\begin{description}
 \item{(i)}\ $0$ is regular for $(0,+\infty)$ and irregular for $(-\infty,0)$, and,  hence, $X$ is initially positive,
\item{(ii)}\  $\lim_{t\to\infty}X_t=-\infty$ a.s.
\end{description}
For (i), see, e.g., Kyprianou \cite{kyprianou14} p. 232. For (ii) notice that  the function $\psi(\lambda):=\lambda-\sqrt{2\lambda}$, cf.  (\ref{ee100}),  satisfies $\psi'(0+)=-\infty$, and, then,  consult  \cite{kyprianou14} p. 233. Consequently, $T$ is almost surely positive and finite.  For the Laplace transform of $T$ we recall the formula
\begin{equation}
\label{e440}
\E_0\lp\e^{-\alpha T
}\rp=\E_0\lp\e^{-\alpha H_0}\rp=1-\frac\alpha{\phi(\alpha)}W^{(\al)}(0) ,
\end{equation}
where $W^{(\al)}$ is the scale function of $X$ and 
\begin{equation*}
\phi(\alpha):=1+\alpha+\sqrt{1+2\alpha}
\end{equation*}
is the inverse of $\lambda\mapsto \psi(\lambda),\,  \lambda\geq \lambda_+,$ with $\lambda_+=2$  the unique positive root of the equation $\psi(\lambda)=0$ (see \cite{kyprianou14} (8.9) p. 234 and for the fact $W^{(\al)}(0)=1$ p. 243). To show that the expression given in (\ref{e440}) coincides with the one in (\ref{e2}) is straightforward. We leave also to the reader to check that the Laplace transform of $f_1$ in (\ref{e3}) is given as in  (\ref{e2}) with $m=1$.\hskip4.7cm
$\square$

\begin{remark}\label{LIL}
The fact that $T_1$ (and then also $T_m$)  is  almost surely  positive and finite can also be proved utilizing the following laws of iterated logarithm
$$
\limsup_{t\to 0}\frac{B_t}{\sqrt{2t\,\ln\ln(1/t)}}=\limsup_{t\to \infty}\frac{B_t}{\sqrt{2t\,\ln\ln t}}=1\quad {\text a.s.}
$$
The first law  implies that there exists a random constant $c_1>0$ such that for all $t\in(0,1)$
$
c_1 t^{3/4}<B_t\leq S_t,
$
and, hence, for all $t<c_1^4$ it holds $t<S_t$ , i.e., $T_1$ is almost surely positive. From  the second law it is seen that there exists a random constant $c_2>1$ such that $
B_t< t^{3/4}
$ for all 
$t>c_2$. Since $S_t=\max\{S_{c_2}, \max_{c_2<s<t} B_s\}$ for $t>c_2$ it follows that $S_t<t^{3/4}<t$ for all $t>\max\{c_2, S_{c_2}^{4/3}\}$ yielding that $T_1$ is almost surely finite.
\end{remark}


\noindent
{\it Proof of Corollary \ref{cor00}.} We consider again the special  case $m=1$. It is possible to integrate the density given in (\ref{e1}) to obtain the expression for the distribution function given in (\ref{e330}). Instead of performing the (tedious) integration we show, firstly,  that the derivative of the right hand side in  (\ref{e330}) equals $f_1$ and, secondly, that the limit as $t\to 0$ equals 1. We have 
$$
f_1'(t)=2g'(t) + g(t)\quad {\rm and}\quad g'(t)=-\frac1{2t} g(t)-\frac 12 g(t),
$$
and, consequently,
\begin{equation}
\label{d030}
f_1'(t)=-\frac 1t g(t).
\end{equation}
Using (\ref{d030}) when differentiating in (\ref{e330}) yields the derivative as given in (\ref{e1}).  Next, notice that
$$
f_1(t)= \sqrt{\frac 2{\pi t}}-1+o(1),\quad t\to 0+
$$ 
and applying this in (\ref{e330}) shows that the limit of the right hand side in (\ref{e330}) equals 1. The moments can be calculated conveniently from the
Laplace transform as given in (\ref{e2a}). We skip the details.
\hskip7cm $\square$

We proceed now with the main result of the paper presenting the joint distribution of $T_m$ and $B_{T_m}$. 

\begin{theorem}
\label{thm1}
The joint density of $T_m$ and $B_{T_m}$ is given by 
\begin{align}\label{4f8}
\nonumber 
\psi_m(t,x):= &\, \P_0\lp T_m\in dt, B_{T_m}\in dx \rp/dtdx\\
= &\,\frac {(\frac tm-x)^+}{mt}\sqrt{\frac 2{\pi t}}\,\exp\lp-\frac {(\frac{2t}m-x)^2}{2t}\rp, \quad  t>0.
\end{align}
Moreover, for $-\frac12\leq m\al<4$ 
\begin{equation}\label{4f80}
\E_0\lp \e^{-\al B_{T_m}}\rp =\frac 1{1-m\al + \sqrt{1+2m\alpha}},
\end{equation}
and, hence, $B_{T_m}$ has all (positive) moments.
\end{theorem}

\noindent 
We state also the following corollary which can be easily verified once recalling that $S_{T_m}=T_m/m.$

\begin{corollary}
\label{EXP}
The joint density of $S_{T_m}-B_{T_m}$ and $S_{T_m}$ is for $ t>0$ and $u>0$ given by 
\begin{align}
\label{EXP1} 
&\P_0\lp S_{T_m}-B_{T_m}\in du, S_{T_m}\in dt \rp\\
&\hskip3cm\nonumber
= \frac {2u}{\sqrt{2\pi (tm)^3}}\,\exp\lp-\frac {(t+u)^2}{2tm}\rp\,dtdu. 
\end{align}
In particular,  $S_{T_m}-B_{T_m}$ is exponentially distributed with mean $m/2$.
\end{corollary}

\begin{remark}\label{LPA}
To explain briefly the heuristics behind the formula (\ref{4f8}) based on the theory 
of spectrally one-sided L\'evy processes and excursions consider 
\begin{align}
\label{f99}
\nonumber
&\hskip -.5cm
\P_0\lp T\in dt , B_T\leq x\rp 
\\&\hskip1cm
=\int_{ z=0}^t\P_0\lp H_0\in dt, B_{\cT_{H_0-}+X_{H_0-}}\leq x, X_{H_0-}\in dz\rp
\end{align}
where $T:=T_1=H_0$ and the identity  $X_{H_0-}=H_0-\cT_{H_0-}$ is used. The joint distribution of $H_0$ and $X_{H_0-}$ can be calculated explicitly. Without going into the details, we state that $X_{H_0-}$ is gamma-distributed  with parameters 2 and 1/2, i.e., 
$$
\P_0\lp X_{H_0-}\in dy\rp= {\displaystyle{ \sqrt{\frac 2\pi}}}\, y^{-1/2}
\,\e^{ -2 y}\, dy,
$$
and 
 the conditional law of $H_0$ given $X_{H_0-} =y>0$ is equal to the law of  $y+\xi^{(1)}_{y}$, where $\xi^{(1)}_{y}$ is the first hitting time of $y$ for  a Brownian motion with drift  $1$ started at 0. To derive (\ref{4f8}) from (\ref{f99}) the conditional law of $B_{\cT_{H_0-}+X_{H_0-}}$ given $H_0$ and $X_{H_0-}$ is needed. Guessing that this conditional distribution is simply given by the It\^o excursion law of a reflecting Brownian motion results into the claimed formula (\ref{4f8}). However, we do not have a rigorous proof of this last statement.
\end{remark}

\noindent  
{\it Proof of Theorem \ref{thm1}} is structured into  several steps starting with  Proposition \ref{prop41} and ending with Proposition \ref{prop44}.
As indicated in Remark \ref{rmk2} it is enough to consider the case $m=1$. Let $T:=T_1$ and   $\psi(t,x):=\psi_1(t,x).$
 Recall that almost surely $0<T<\infty$, and, hence, also $0<S_T<\infty$  almost surely.
 
 \begin{proposition}
\label{prop41}
For $t>0$
 \begin{align}\label{4f1}
\psi(t,x)=\sqrt{\frac 2{\pi }}(t-x)^+A_1(t,x),
\end{align}
where the function $A_1$ is given for all $s>0$ and $s>y$ by
$$
A_1(s,y):=\E_0\lp (s-H_s)^{-3/2}\,\exp\lp-\frac {(s-y)^2}{2(s -H_s)}\rp {\bf 1}_{\{s<S_T\}}\rp.
$$
 \end{proposition}

 \begin{proof} Our approach is similar to the one  presented in  Rogers \cite{rogers81}. 
 Notice first that $A_1$ is well defined since $S_T>s$ implies that $s>H_s$.  Let $h$ be a test function and consider 
  \begin{align}\label{4f2}
& h(T,B_T)\\
 &\hskip1cm\nonumber
 =\sum_{\al>0} {\bf 1}_{\{\zeta(e_\al)>0\}}h\big(\al,\al-e_\al(\al-H_\al)\big)
 {\bf 1}_{\{H_\al\leq\al\leq H_\al +\zeta(e_\al)\, ;\,\al\leq S_T\}},
\end{align} 
where $H_\al:=\inf\{ u\geq 0\,;\, B_u=\al\},$  $H_{\al+}:=\inf_{y>\al} H_y=\inf\{ u\geq 0\,;\, B_u>\al\},$ and
$$
e_\al(u):=\al-B_{H_\al+u}\quad {\rm for}\quad 0\leq u\leq 
\zeta(e_\al):=H_{\al+}-H_\al.
$$
Notice that the sum in (\ref{4f2}) contains only one term and this is connected to the excursion straddling $T$. If $S_T=\al$ then
$$
B_T=S_T-(S_T-B_T)=\al-e_\al(\al-H_\al).
$$
Let $n^+$ denote the characteristic measure of the Poisson point process associated with  the excursions of reflecting Brownian motion. Then we have for $z>0$  the formula, see Salminen, Vallois and Yor \cite{salminenvalloisyor07} Theorem 2,
$$
n^+\lp e(u)\in dy, \zeta(e)>u\rp = \frac {2y}{\sqrt{2\pi u^3}}\,\e^{-y^2/2u}\, dy.
$$ 
 \noindent
Taking the expectation in (\ref{4f2}) and using the Master Formula for Poisson point processes, see Revuz and Yor  \cite{revuzyor01} p. 471, yield
  \begin{align*}\label{4f3}
& \E_0\lp h(T,B_T)\rp\nonumber\\
 &\hskip.5cm\nonumber
= \sqrt{\frac {2}{{\pi}}}\,\E_0\lp\int_{0}^\infty ds\int_{0}^\infty dy \, h(s,s-y)  \frac {y}{ (s-H_s)^{3/2}}\,\e^{-y^2/2(s-H_s)}
\, {\bf 1}_{\{s<S_T, H_s<s\}}\rp\\
 &\hskip.5cm\nonumber
= \sqrt{\frac {2}{{\pi}}}\,\E_0\lp\int_{0}^\infty ds\int_{-\infty}^s dz \, h(s,z)  \frac {s-z}{ (s-H_s)^{3/2}}\,\e^{-(s-z)^2/2(s-H_s)}
\, {\bf 1}_{\{s<S_T\}}\rp,
  \end{align*}
where in the second step we have substituted $z=s-y$ and used the fact that  $S_T>s$ implies $H_s<s$.  Formula (\ref{4f1}) follows now immediately.  
 \end{proof}
 
 To proceed we write for $s>y$
 \begin{equation}
 \label{A1}
  A_1(s,y)= A_2(s,y)- A_3(s,y),
   \end{equation}
 where 
 $$
A_2(s,y):=\E_0\lp (s-H_s)^{-3/2}\,\exp\lp-\frac {(s-y)^2}{2(s -H_s)}\rp {\bf 1}_{\{H_s<s\}}\rp. 
 $$
 and
  \begin{equation}
\label{A3}
 A_3(s,y):=\E_0\lp (s-H_s)^{-3/2}\,\exp\lp-\frac {(s-y)^2}{2(s -H_s)}\rp {\bf 1}_{\{H_s<s, S_T< s\}}\rp.
   \end{equation}
 In fact, we need a slightly more general functional than $A_2$ and, hence,  introduce for $s>0, u>0$ and $v\leq u$ 
  $$
A_4(s,u,v):=\E_0\lp (u-H_v)^{-3/2}\,\exp\lp-\frac {s^2}{2(u -H_v)}\rp {\bf 1}_{\{H_v<u\}}\rp. 
 $$
 
 \begin{lemma}
 \label{lemma41}
 It holds
  \begin{equation*}
A_4(s,u,v)=\frac{s+v}{su^{3/2}}\,\exp\lp-\frac {(s+v)^2}{2u}\rp. 
    \end{equation*}
 In particular,
 \begin{equation}\label{A22}
A_2(s,y)=A_4(s-y,s,s)=  \frac{2s-y}{(s-y)s^{3/2}}\,\exp\lp-\frac {(2s-y)^2}{2s}\rp. 
    \end{equation} 
\end{lemma}

\begin{proof} 
Recall that for $v>0$ 
$$
\P_0(H_v\in dt)=\frac{v}{\sqrt{2\pi t^3}}\,\e^{- v^2/2t}\,dt 
$$ 
and, consequently,
  $$
A_4(s,u,v)=\int_0^u \frac 1{\sqrt{(u-t)^{3}}}\,\e^{- s^2/2(u-t)} \,\frac{v}{\sqrt{2\pi t^3}}\,\e^{- {v^2}/{2t}}\, dt. 
 $$
Substituting $t=u/(1+r)$ yields after some manipulations
  \begin{align}\label{4f3} 
&
A_4(s,u,v)\\
&\hskip1cm\nonumber 
 =\frac v{u^2}\exp\lp - \frac {s^2}{2u} -\frac{v^2}{2u}\rp \int_0^\infty  \frac {1+r}{\sqrt{2 \pi r^3}}
\exp\lp - \frac {v^2}{2u}r -\frac{s^2}{2ur}\rp\,dr.
     \end{align}
In the integral term above, we identify the following Laplace transforms
$$
\int_0^\infty  \frac {1}{\sqrt{2 \pi r^3}}
\exp\lp - \frac {v^2}{2u}r -\frac{s^2}{2ur}\rp\,dr= \frac{\sqrt{u}}{s}\, \exp\lp-\frac{vs}u\rp,
$$
i.e. the Laplace transform of the first hitting time, and
  \begin{equation}\label{B0}
\int_0^\infty  \frac {1}{\sqrt{2 \pi r}}
\exp\lp - \frac {v^2}{2u}r -\frac{s^2}{2ur}\rp\,dr= \frac{\sqrt{u}}{v}\, \exp\lp-\frac{vs}u\rp, 
\end{equation}
 i.e., the Green kernel of the standard Brownian motion. Putting these expressions in (\ref{4f3}) yields the claimed formula.
   \end{proof}

Next we derive an alternative expression for the function $A_3$ crucial for the further analysis. 
 \begin{lemma}
 \label{lemma42}
 For $s>y$ it holds
  \begin{equation}\label{A33}
 A_3(s,y)=\frac 1{s-y}\,\E_0\lp \frac {2s-y-B_T}{(s-T)^{3/2}}\,\exp\lp-\frac {(2s-y-B_T)^2}{2(s -T)}\rp {\bf 1}_{\{T< s\}}\rp.
    \end{equation}
 
\end{lemma}

\begin{proof} 
In the definition (\ref{A3}) of $A_3$ we have the condition $S_T<s$. Consequently, because $T=S_T$, it holds on $\{S_T<s\}$ that
$$
H_s=\inf\{u\geq T\,;\, B_u=s\}=T+H'_{s-B_T},
$$
where $H'_x=\inf\{u\,;\, B'_u=x\}$ and $B'_u:=B_{T+u}-B_T,\, u\geq 0,$ is a Brownian motion independent of $ (B_u)_{0\leq u\leq T}$. Consider now
  \begin{align*} 
&
A_3(s,y)=\E_0\lp\E_0\lp(s-H_s)^{-3/2}\,\exp\lp-\frac {(s-y)^2}{2(s -H_s)}\rp {\bf 1}_{\{H_s<s, S_T< s\}}\,\Big | \, \cF_T\rp\rp \\
&\hskip1.45cm =\E_0\lp A_4(s-y,s-T,s-B_T){\bf 1}_{\{T<s\}}\rp.
 \end{align*}
Using the expression for $A_4$ given in Lemma \ref{lemma41} results into the claimed formula (\ref{A33}).
\end{proof}

Recall that $\psi$ denotes the density of  $(T,B_T)$. Clearly, if we know $\psi$, it is seen from Lemma \ref{lemma42} that we can calculate $A_3$. This observation leads to the following property of $\psi$.  
 \begin{proposition}
 \label{prop42}
 The density function $\psi$ satisfies 
   \begin{equation}\label{ieq}
 \psi=\psi_0-\Lambda\psi,
     \end{equation}
     where for $t>0$ and $x<t$
       \begin{align*} 
&     \psi_0(t,x):=\sqrt{\frac 2{\pi }}(t-x)^+A_2(t,x)
= \sqrt{\frac 2{\pi }}\, \frac{2t-x}{t^{3/2}}\,\exp\lp-\frac {(2t-x)^2}{2t}\rp,
 \end{align*}
 and 
       \begin{align*} 
& \hskip-.5cm  \Lambda\psi(t,x):= \sqrt{\frac 2{\pi }}(t-x)^+ A_3(t,x)\\ 
 & \hskip1.1 cm = \sqrt{\frac 2{\pi }} \int_0^tdu\int_{-\infty}^u dv \ 
 \frac {2t-x-v}{(t-u)^{3/2}}\,\exp\lp-\frac {(2t-x-v)^2}{2(t -u)}\rp \psi(u,v).
 \end{align*} 
 
 \end{proposition}
 \begin{proof}
 The claim follows by exploiting (\ref{4f1}), (\ref{A1}), (\ref{A22}), and (\ref{A33}).
 \end{proof}

Inspired by (\ref{ieq}) we study the integral equation
   \begin{equation}\label{ieq2}
 h=\psi_0-\Lambda h
     \end{equation}
for measurable functions $h:\cD\mapsto \R_+$ with  $\cD:=\{(t,x)\,;\, t>0, x<t\}$. In the proposition  to follow it is seen that our candidate for the density of $(T,B_T)$ solves (\ref{ieq2}).
 \begin{proposition}
 \label{prop43}
The function 
   \begin{equation*}
\psi^*(t,x):= \frac {(t-x)^+}{t}\sqrt{\frac 2{\pi t}}\,\exp\lp-\frac {({2t}-x)^2}{2t}\rp,\quad   t>0,
   \end{equation*}
  is a density function and  solves the integral equation (\ref{ieq2}).
 \end{proposition}
  \begin{proof}
The claims can be accomplished by straightforward (but tedious) integrations. We skip these calculations. 
 \end{proof}
 
 Our final goal is to show that the integral equation (\ref{ieq2}) has an integrable and almost everywhere  unique  solution. For this we need the following result concerning the operator $\Lambda$.   
 
  \begin{lemma}
 \label{lemma43}
 Let $\lambda\geq 0$ and $h:\cD\mapsto \R_+$ be measurable. Then
   \begin{equation*}
   \int_\cD \e^{-\lambda t}\,\Lambda h(t,x)dtdx= \frac 2{\sqrt{1+2\lambda}}\int_{\cD}
  \, \e^{-\lp\lambda u+(u-v)(1+\sqrt{1+2\lambda})\rp}\,
    h(u,v)dudv.
   \end{equation*}

 \end{lemma}
 \begin{proof}
 Using Fubini's theorem we get 
    \begin{equation*}
   \int_\cD \e^{-\lambda t}\,\Lambda h(t,x)dtdx= \int_{\cD} \rho(u,v)\,h(u,v)dudv,
   \end{equation*}
   where 
   $$
   \rho(u,v):=  \sqrt{\frac 2{\pi }} \int_\cD \ \e^{-\lambda t}\, 
 \frac {2t-x-v}{(t-u)^{3/2}}\,\exp\lp-\frac {(2t-x-v)^2}{2(t -u)}\rp {\bf 1}_{\{u<t\}}dtdx.
 $$ 
 To check that $\rho$ takes the claimed form, set $t-u=r$, integrate first with respect to $x$, and use then  (\ref{B0}) with $u=1$, $v^2=2\la+1$, and $s=u-v$.
 \end{proof}
 
  \begin{proposition}
 \label{prop44}
 The integral equation (\ref{ieq2}) has an integrable and almost everywhere unique  solution.
  \end{proposition} 
\begin{proof}
Let $\phi_1$ and $\phi_2$ be two integrable non-negative solutions. Then $\phi:=\phi_1-\phi_2$ solves $\phi=-\Lambda\phi$, and it holds 
$$
  \int_\cD \e^{-\lambda t}\,|\Lambda \phi(t,x)|dtdx\leq    \int_\cD \e^{-\lambda t}\,\Lambda( |\phi|)(t,x)|dtdx.
  $$
By Lemma \ref{lemma43} with $h=|\phi|$
\begin{align*} 
 \int_\cD \e^{-\lambda t}\,\Lambda( |\phi|)(t,x)|dtdx&= \frac 2{\sqrt{1+2\lambda}}\int_{\cD}
  \, \e^{-\lp\lambda u+(u-v)(1+\sqrt{1+2\lambda})\rp}\,|\phi(u,v)|\,dudv\\
  &
 \leq\frac 2{\sqrt{1+2\lambda}}\int_{\cD}
  \, \e^{-\lambda u}\,|\phi(u,v)|\,dudv, 
\end{align*}
where in the second step it is used that $u>v$ inside the integral. Choosing $\lambda$ so that
$$
\frac 2{\sqrt{1+2\lambda}}\leq \frac 12
$$
and  recalling that  $\phi=-\Lambda\phi$ we obtain 
$$
\int_{\cD}
  \, \e^{-\lambda u}\,|\phi(u,v)|\,dudv\leq \frac 12 \int_{\cD}
  \, \e^{-\lambda u}\,|\phi(u,v)|\,dudv,
  $$
  i.e., $\phi\equiv 0$ almost everywhere, as claimed. 
\end{proof}
To conclude, we have proved that 1) the density function of $(T,B_T)$ solves the integral equation (\ref{ieq2}), 2) also the candidate density function given in (\ref{4f8}) solves this equation, and 3) the equation has an almost everywhere unique solution. Consequently, the function given in (\ref{4f8}) is the density of 
$(T,B_T)$. To calculate the  Laplace transform of $B_{T}$ is a straightforward but tedious  integration, and we skip the details. The proof of Theorem \ref{thm1} is now complete. \hskip11.5cm$\square$

\medskip
\section{Joint distribution of $T$ and $B_T$ for Brownian motion with drift}
\label{sec22}

In this section, using 
Girsanov's theorem,  we derive the joint distribution of  $T_m$ and $B_{T_m}$ for a Brownian motion with drift $\mu$.  
We let $\P^{(\mu)}_x$ and   $\E^{(\mu)}_x$ denote the probability measure and the expectation of a Brownian motion with drift $\mu$ when initiated from $x.$ Under  $\P^{(\mu)}_x$ and   $\E^{(\mu)}_x$ the notation $(B_t)_{t\geq 0}$ stands for a Brownian motion with drift $\mu$. We also write $\P_x$ instead of $\P^{(0)}_x$. 



\begin{theorem}
\label{thm2}
For Brownian motion with drift $\mu$ the joint distribution of $T_m$ and $B_{T_m}$ is given by
\begin{align}
\label{f14}
\nonumber
&
\P^{(\mu)}_0\lp T_m\in dt, B_{T_m}\in dx, T_m<\infty \rp
\\
&\hskip2.8cm
={\rm e}^{\,\mu x-\frac{\mu^2t}{2}}\ \frac {(\frac tm-x)^+}{mt}\sqrt{\frac 2{\pi t}}\,\exp\lp-\frac {(\frac{2t}m-x)^2}{2t}\rp dtdx
\end{align}
In particular, for $\mu m\not= -1$
\begin{align}
\label{e11}
\nonumber
\hskip-2.8cm\P^{(\mu)}_0(T_m\in dt)=\frac 1 {m^2}\,\e^{2\mu t/m} \Big( |1+\mu m|
f_1\lp\lp1+\mu m\rp^2t/m^2\rp &\\
\hskip4.8cm
+2(-\mu m-1)^+\Big) dt,
\end{align}
where
\begin{equation}
\label{e111}
f_1(t):=
\P_0(T_1\in dt)/dt=\sqrt{\frac 2{\pi t}}\lp\e^{-t/2}-\int_t^\infty\e^{-y^2/2t} dy\rp,
\end{equation}
and for  $\mu m= -1$
\begin{equation}\label{e110}
\hskip-.5cm\P^{(\mu)}_0(T_m\in dt)=
\frac 1 m\sqrt{\frac 2{\pi t}}\,\e^{-2t/m^2}dt. 
\end{equation}
Moreover, it holds
\begin{equation}
\label{e22}
\P^{(\mu)}_0(T_m<\infty)=
\begin{cases}
1,& \mu m\leq 1,\\
\displaystyle{\frac 1{\mu m}}, & \mu m>1.
\end{cases}
\end{equation}
The Laplace transform of $T_m$ is for $\alpha> -(1-\mu m)^2/(2m^2)$ in case $\mu m\geq -1$ and for $\alpha> 2\mu/m$ in case $\mu m\leq -1$ given by
\begin{equation} \label{e21}
\E^{(\mu)}_0\lp\e^{-\alpha T_m}{\bf 1}_{\{T_m<\infty\}}\rp=\frac{2}{1+\mu m+\sqrt{(1-\mu m)^2+2m^2\alpha}}.
\end{equation}
The Laplace transform of $B_{T_m}$ on $\{T_m<\infty\}$ is for $$-\frac{(1-\mu m)^2}2< m\al<
\mu m+2+\sqrt{( \mu m)^2+4}$$  given by
\begin{equation}\label{4f801}
\E_0^{(\mu)}\lp \e^{-\al B_{T_m}}{\bf 1}_{\{T_m<\infty\}}\rp =\frac 2{1+\mu m -m\al + \sqrt{(1-\mu m) ^2+2m\alpha}}.
\end{equation}
\end{theorem}

In the  proof of the next corollary one can make use of the proof of Corollary \ref{cor00}; in particular formula(\ref{d030}). We skip the details. 

 \begin{corollary}\label{cor11}
 The distribution function of $T_m$  is given by: 
 \begin{enumerate}
 \item if $\mu m \not\in \{-1,0,1\}$
 \begin{align}
\label{d00}
&
\P^{(\mu)}_0\lp T_m>t\rp
\\
&\nonumber
\hskip.5cm
=
\frac 1{2\mu m}F(t;\mu,m)
+ \frac1{\mu m}(\mu m -1)^+- \frac1{\mu m}(-\mu m -1)^+\e^{2\mu t/m},
\end{align}
where
$$
F(t;\mu,m):=
\vert \mu m -1\vert f_1\big((1-\mu m)^2t/m^2\big)-|\mu m+1|\, \e^{2\mu t/m}\,f_1\big((1+\mu m)^2t/m^2\big),
$$
\item if $\mu=0$
\begin{equation}
\label{d01}
\P^{(\mu)}_0\lp T_m>t\rp=1-G\lp\frac t{m^2}\rp-\frac t{m^2}f_1\lp\frac t{m^2}\rp,
\end{equation}
where $G$ is as in Remark \ref{rmk22},
\item if $\mu m=1$
\begin{equation}
\label{d02}
\P^{(\mu)}_0\lp T_m>t\rp=\frac  m{\sqrt{2\pi t}}-f_1\lp\frac {4t}{m^2}\rp\,\e^{2t/m^2},
\end{equation}
\item if $\mu m=-1$
\begin{equation}
\label{d03}
\P^{(\mu)}_0\lp T_m>t\rp=\frac  m{\sqrt{2\pi t}}\,\e^{-2t/m^2}-f_1\lp\frac {4t}{m^2}\rp.
\end{equation}
\end{enumerate}
\end{corollary}

Also the proof of the next corollary is straightforward, and we skip the details. Recall that $\mu m\leq 1$ implies that $T_m<\infty$ almost surely. It is a bit surprising that the distribution of $S_{T_m}-B_{T_m}$ does   not depend explicitly on~$\mu$.
\begin{corollary}
\label{EXP2}
The joint density of $S_{T_m}-B_{T_m}$ and $S_{T_m}$ is for $ t>0$ and $u>0$ given by 
\begin{align}
\label{EXP12} 
&\P_0^{(\mu)}\lp S_{T_m}-B_{T_m}\in du,S_{T_m}\in dt , T_m<\infty\rp \\
&\hskip3cm
\nonumber
={\rm e}^{\,\mu (t-u)-\frac{\mu^2tm}{2}}\  \frac {2u}{\sqrt{2\pi (tm)^3}}\,\exp\lp-\frac {(t+u)^2}{2tm}\rp\,dtdu. 
\end{align}
In particular,  $S_{T_m}-B_{T_m}$ is, when conditioned on $T_m<\infty$,  exponentially distributed with mean $m/2$.
\end{corollary}

\noindent
{\it{Proof of Theorem \ref{thm2}.}}
For notational simplicity we prove the result for $m=1$ and let $T:=T_1.$ The proof is easily modified for a general $m>0.$ Alternatively, one could use the scaling property of Brownian motion with drift, which says that the $\P^{(\mu)}_0$-law of $(T_m, B_{T_m})$ is equal to  the $\P^{(\mu m)}_0$-law of $(m^2T_1, mB_{T_1})$. Let now $\varphi:\R_+\times\R\mapsto \R_+$ be a Borel measurable and  bounded function. Then for $n>0$
$$
\E^{(\mu)}_0\lp\varphi(T, B_T)\rp=\lim_{n\to\infty} \E^{(\mu)}\lp\varphi(T, B_T){\bf 1}_{\{T\leq n\}}\rp.
$$
Clearly,
\begin{align*}
&\Delta_n:=\E^{(\mu)}_0\lp\varphi(T, B_T){\bf 1}_{\{T\leq n\}}\rp
\\
&
\hskip0.70cm
=\E^{(\mu)}_0\lp\varphi(T\wedge n, B_{T\wedge n}){\bf 1}_{\{T\leq n\}}\rp
\\
&
\hskip0.70cm
=\E_0\lp \varphi(T\wedge n, B_{T\wedge n}){\bf 1}_{\{T\leq n\}}M_{T\wedge n}\rp,
\end{align*}
where $(M_t)_{t\geq 0}$ is  $\P_0$-martingale given by
$
M_t:=\exp\lp\mu B_t-\frac{\mu^2}2t\rp
$,
and in the third step Girsanov's theorem is used which is applicable since
$\varphi(T\wedge n, B_{T\wedge n}){\bf 1}_{\{T\leq n\}}$ is $\cF_{T\wedge n}$-measurable.
Consequently, Theorem \ref{thm1} yields
$$
\Delta_n=\int_{\R_+\times\R}\varphi(t,x)\,{\rm e}^{\,\mu x-\frac{\mu^2t}{2}}\lp 1-\frac xt\rp^+\sqrt{\frac 2{\pi t}}\,\e^{-\lp x-2t\rp^2/2t} {\bf 1}_{\{t<n\}}dt dx,
$$
and this proves (\ref{f14}) as $n\to\infty$. The $\P^{(\mu)}_0$-density of $T$ is obtained by integrating in (\ref{f14}) over $x$. For this consider (some details are omitted)
\begin{align*}
&
\P^{(\mu)}_0(T\in dt)/dt=\int_\R{\rm e}^{\,\mu x-\frac{\mu^2t}{2}}\lp 1-\frac xt\rp^+\sqrt{\frac 2{\pi t}}\,\e^{-\lp x-2t\rp^2/2t} dx
\\
&\hskip2.8cm={\rm e}^{\,2\mu t}\sqrt{\frac 2{\pi t}}\int_{-\infty}^t \lp 1-\frac xt\rp\,\e^{-\lp x-(\mu+2)t\rp^2/2t}dx
\\
&\hskip2.8cm
=-{\rm e}^{\,2\mu t}\sqrt{\frac 2{\pi t}}\int_{-\infty}^{-(\mu+1)t} \lp \mu+1+\frac yt\rp\,
\e^{-y^2/2t} dy,
\end{align*}
and we obtain the claimed formulas (\ref{e11}) and (\ref{e110}) (in case $m=1$) after some (tedious) integrations.  Knowing the Laplace transform of the $\P_0$-density of $T$, see (\ref{e2}),  it is  fairly straightforward to calculate from (\ref{e11}) and  (\ref{e110}) the  transform of the corresponding $\P^{(\mu)}_0$-density and also to deduce   (\ref{e22}). To derive the formula (\ref{4f801}) demands also tedious integrations. We leave the details to the reader. \hskip10.8cm $\square$


\medskip

In case  $\mu<0$ it holds  that $S_\infty:=\lim_{t\to\infty}S_t<\infty$.  Recall also that $S_\infty$ is in this case under $\P^{(\mu)}_0$ exponentially distributed with parameter  $2|\mu|$. Let
$
\rho:=\inf\{t; B_t=S_\infty\}
$. 
We are interested in  decomposing  the joint distribution of $T_m$ and $B_{T_m}$ into two parts  depending
on whether $T<\rho$ or $T>\rho$. A crucial tool in our analysis is  the following description of the conditional law of  Brownian motion with negative drift given the value of the global supremum, see  Williams \cite{williams74}.
\begin{theorem}
\label{wil}
For $\mu<0$ and conditionally on $S_\infty=x$ the process $(B_t)_{0\leq t<\rho}$
is  under $\P^{(\mu)}_0$ distributed as $(B_t)_{0\leq t<H_x}$
under $\P^{(|\mu|)}_0$. In other words, for a bounded measurable  functional $F$ on truncated paths and a bounded measurable function $h$
\begin{align}
\label{wil1}
\nonumber
&\hskip0cm
\E^{(\mu)}_0\big(F(B_u;0\leq u<\rho)h(S_\infty)\big)\\
&\hskip2cm
=2|\mu|\int_0^\infty \e^{-2|\mu| x}\, h(x) \,
\E^{(|\mu|)}_0\big( F(B_u;0\leq u<H_x)\big)\, dx.
\end{align}
\end{theorem}

\noindent
In the next theorem we give the joint distribution under the restriction ~$T_m<\rho$.

\begin{theorem}
\label{thm3}
For Brownian motion with negative drift  $\mu<0$  it holds
\begin{align}
\label{f144}
&
\P^{(\mu)}_0\lp T_m\in dt, B_{T_m}\in dx, T_m<\rho \rp
\\
\nonumber
&\hskip2.8cm
={\rm e}^{\,|\mu| x-\frac{\mu^2t}{2}-\frac{2|\mu|t}m}\ \frac {(\frac tm-x)^+}{mt}\sqrt{\frac 2{\pi t}}\,\exp\lp-\frac {(\frac{2t}m-x)^2}{2t}\rp dtdx.
\end{align}
In particular, with $f_1$ is as given in (\ref{e111})
\begin{equation}
\label{e114}
\P^{(\mu)}_0(T_m\in dt, T_m<\rho)=\frac 1{m^2}(1-\mu m) f_1((1-\mu m)^2t/m^2)dt
\end{equation}
and
\begin{equation}
\label{e224}
\P^{(\mu)}_0(T_m<\rho)=\displaystyle{\frac 1{1-\mu m}}.
\end{equation}
\end{theorem}

\begin{proof} Again, we prove the result for $T:=T_1$. Consider for a bounded and measurable $h:\R_+\times\R\mapsto\R_+$
\begin{align*}
&\Delta:=\E^{(\mu)}_0\lp h(T,B_T); T<\rho\rp\\
&\hskip.5cm
=2|\mu|\int_0^\infty \e^{-2|\mu| y}\,\E^{(|\mu|)}_0\lp h(T,B_T); T<H_y\rp\,dy\\
&\hskip.5cm
=2|\mu|\int_0^\infty \e^{-2|\mu| y}\,\E^{(|\mu|)}_0\lp h(T,B_T); T<y\rp\, dy,
\end{align*}
where, in the first step,  (\ref{wil1}) is used, and for  the second step observe that
\begin{equation}
\label{e225}
\{T<H_y\}=\{S_T<y\}=\{T<y\}.
\end{equation}
Consequently, we may apply (\ref{f14}) in Theorem \ref{thm2} to obtain
\begin{align*}
&\Delta=2|\mu|\int_0^\infty \e^{-2|\mu| y}\lp\int_{\R_+\times\R} h(t,z)\,
\P^{(|\mu|)}_0\lp T\in dt, B_T\in dz, T<y\rp\rp dy\\
&\hskip.3cm
=\int_{\R_+\times\R} h(t,z)\,\sqrt{\frac 2{\pi t}}\,\lp 1-\frac zt\rp^+\,
{\rm e}^{\,|\mu| z-\frac{\mu^2t}{2}-\lp z-2t\rp^2/2t}\lp\int_t^\infty2|\mu|\e^{-2|\mu| y}dy\rp dtdz,
\end{align*}
from which (\ref{f144}) is easily deduced. Statements (\ref{e114}) and (\ref{e224}) can be verified by straightforward integrations -- we omit the details.
\end{proof}

 The results analogous to the results in Theorem \ref{thm3} under the restriction $T\geq \rho$  can now be obtained by ``subtracting the formulas in Theorem \ref{thm3} from the corresponding formulas in Theorem \ref{thm2}''. For instance, for $\mu<0$  
\begin{align*}
\P^{(\mu)}_0( &T_m\in dt, \rho \leq T_m)/dt\\
&=\frac 1 {m^2}\,\e^{2\mu t/m} \Big( |1+\mu m|
f_1\lp\lp1+\mu m\rp^2t/m^2\rp 
+2(-\mu m-1)^+\Big)\\ 
&\hskip3cm -\frac 1{m^2}(1-\mu m) f_1((1-\mu m)^2t/m^2).
\end{align*}

\section{Path transformations}
\label{sec31}
\begin{figure}[h!]
\centering
\includegraphics[scale=0.45]{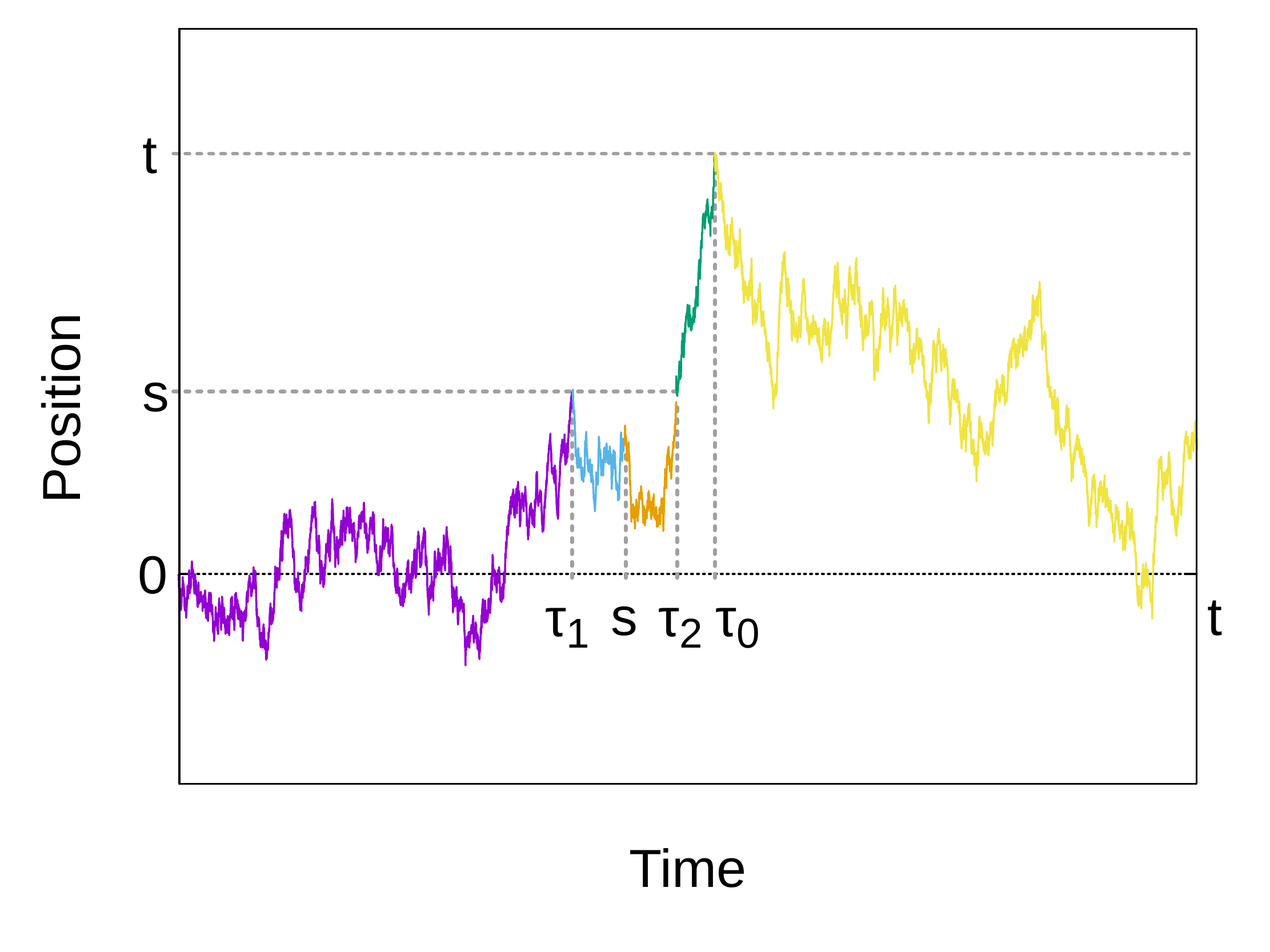}\hfill \includegraphics[scale=0.45]{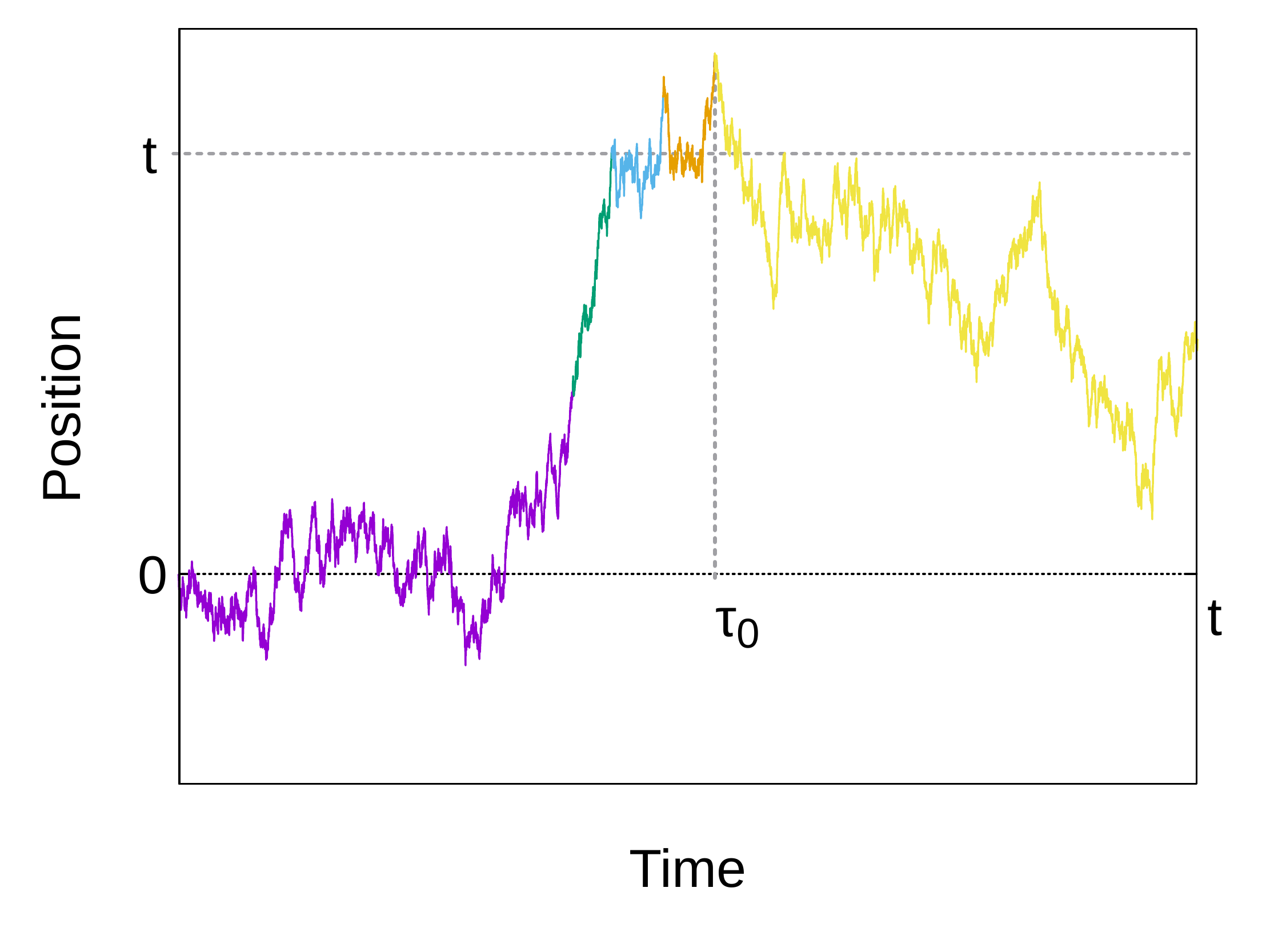}
\caption{\label{fig2}Transformation of a sample path $B_{<t}$ with $S_t=t$, $B_t=x$ and $t>T=s>0$ (top), into a sample path with $S_t=t+u>t$ and $B_t=x+u$ (bottom). For the sake of simplicity here we focus on $S_t$ only, but the transformation works simultaneously and adequately on $B_t$, as explained in the main text.}
\end{figure}

Another approach through which it seems possible to construct the joint density of $\lp T,B_T\rp$ as given in~(\ref{4f8}) is with path transformations. We sketch the idea for $\mu=0$ and $m=1$.
\medskip

The starting point here is the observation that $\lb T \in dt, B_T \in dx \rb$ is (strictly) contained in $\lb S_t \in dt, B_t \in dx \rb$. Let us therefore introduce two sets of sample paths ---that is, two subsets of $\cC(\RR_+)$, the set of continuous functions on $\RR_+$. For $t>0$, $x\leq t$, and $u>0$ we define
\begin{equation*}
\Gamma(t,x;u)=\left\lbrace \omega \in \cC(\RR_+)\,\big\vert \,\omega(0)=0,\, \sup_{r\in [0,u]} \omega(r) \in dt\mathrm{\ and\ }\omega(u) \in dx \right\rbrace,
\end{equation*}
and
\begin{equation*}
\Gamma_o(t,x;u)= \left\lbrace \omega \in \cC(\RR_+)\, \big\vert \,\omega \in \Gamma(t,x;u)\mathrm{\ and\ } \exists\, s < u, \sup_{r\in [0,s]} \omega(r) < s \right\rbrace.\nonumber
\end{equation*}
As noted above, $\Gamma_o(t,x;u)\subsetneq \Gamma(t,x;u)$, and the event $\left\lbrace T\in dt , B_T\in dx\right\rbrace$ corresponds to $\omega \in \Gamma(t,x;t) \setminus \Gamma_o(t,x;t)$. Hence, if we write $B_{<u}$ for the sample path of $(B_s)_{0\leq s \leq u}$, we have heuristically
\begin{equation}
\P_0\lp T\in dt , B_T\in dx\rp =\P_0 \lp B_{<t}\in\Gamma(t,x;t)\rp -\P_0 \lp B_{<t}\in\Gamma_o(t,x;t) \rp.\label{EqGP}
\end{equation}
The first term on the right-hand side in~(\ref{EqGP}) may simply be identified with the joint distribution of $S_t$ and $B_t$, which of course is well known~\cite{levy48} and, for $y\geq 0$ and $x \leq y$, is given by
\begin{align}
\P_0 \lp B_{<t}\in\Gamma(y,x;t)\rp &=\P_0\lp S_t \in dy,\ B_t \in dx \rp \nonumber\\
&= \sqrt{\frac{2}{\pi t^3}}\,(2y-x)\,e^{-(2y-x)^2/2t}dxdy.\label{EqJMP}
\end{align}
\medskip

\noindent It is the second term on the right-hand side in~(\ref{EqGP}) that we propose to compute via path transformations between $\Gamma_o(t,x;t)$ and $\bigcup_{u > 0}\Gamma(t+u,x+u;t)$. Combined with (\ref{EqGP}) and (\ref{EqJMP}), this correspondence will lead to
\begin{align}
\P_0\lp T\in dt , B_T\in dx\rp&=\P_0\lp B_{<t}\in\Gamma(t,x;t)\rp -\P_0 \lp B_{<t}\in\Gamma_o(t,x;t) \rp\nonumber \\
&\hspace*{-2.5cm}= \P_0 \lp B_{<t}\in\Gamma(t,x;t)\rp - \P_0 \lp B_{<t}\in\bigcup_{u > 0}\Gamma(t+u,x+u;t) \rp \nonumber \\
&\hspace*{-2.5cm}= \Big(\sqrt{\frac{2}{\pi t^3}}\,(2t-x)\,e^{-(2t-x)^2/2t} \nonumber \\
&\hspace*{-2.5cm}\phantom{=}-\,\int_0^\infty\,\sqrt{\frac{2}{\pi t^3}}\;[2(t+u)-(x+u)]\;e^{-[2(t+u)-(x+u)]^2/2t}\,du\Big) dtdx\nonumber\\
&\ \nonumber\\
&\hspace*{-2.5cm}= \left(1-\frac{x}{t}\right)\;\sqrt{\frac{2}{\pi t}}\;e^{-(2t-x)^2/2t}dtdx,\label{JL}
\end{align}
which is the same as $\psi_1(x,t)$ in equation~\ref{4f8}.

We show on Figure~\ref{fig2} a procedure that indeed transforms a path $\omega \in \Gamma_o(t,x;t)$ into a path $\omega \in \bigcup_{u > 0}\Gamma(t+u,x+u;t)$. There remains to prove that this transformation is bijective or at least that it allows us to assert
\begin{equation}
\P_0 \lp B_{<t}\in\Gamma_o(t,x;t) \rp=\P_0 \lp B_{<t} \in \bigcup_{u > 0}\Gamma(t+u,x+u;t) \rp.
\end{equation}
The main idea behind this transformation is that if $T<t$ while $S_t=t$, then $\exists s<t, S_s=s$ while $S_t=t$ ---and this means that there exists a downward excursion away from $S_s=s$ straddling $s$. This excursion may be extracted and used to transform the path.

The cutting times that are needed to transform an initial path $\omega \in \Gamma_o(t,x;t)$ are also shown on figure~\ref{fig2}, and they are well defined:
\begin{itemize}
\item $\tau_1$ is the first time when level $s$ is hit: $\tau_1=\inf \left\lbrace r>0, \omega(r)=s\right\rbrace$ (it is guaranteed to exist since $\sup_{r\in [0,t]} \omega(r) \in dt$ and $t>s$),
\item $\tau_2$ is the first time when level $s$ is hit \textit{after} $\tau_1$: $\tau_2=\inf \left\lbrace r>\tau_1, \omega(r)\geq s\right\rbrace$ (it is guaranteed to exist since $\omega \in \Gamma_o(t,x;t)$ and $\sup_{r\in [0,s]} \omega(r) \in ds$ and $t>s$),
\item $\tau_0$ is the time when $\sup_{r\in [0,t]} \omega(r) =t$ is set: $\tau_0=\inf \left\lbrace r\geq 0, \omega(r)=t\right\rbrace$.
\end{itemize}
\medskip

Note that $\tau_1 \leq s \leq \tau_2 < \tau_0$. Finally, the transformation shown in the figure  may be summarized as follows:
\begin{enumerate}
\item extract the downward excursion between $\tau_1$ and $\tau_2$;
\item bring ``forward'' (to $\tau_1$) the $\left[\tau_2,\tau_0\right]$ part;
\item insert immediately afterwards the excursion transformed into an (upward) first passage bridge~\cite{bcp03};
\item insert the final, post-$\tau_0$ part shifted upward as needed (namely, by a distance $u=2\lp s-\omega(s)\rp$).
\end{enumerate}    

\bigskip
\noindent
{\bf Acknowledgements.} Paavo Salminen thanks Magnus Ehrnrooths stiftelse for financial support.

\bibliographystyle{plain}
\bibliography{biblio}

\begin{thebibliography}{10}

\bibitem{bcp03}
J.~Bertoin, L.~Chaumont, and J.~Pitman.
\newblock Path transformations of first passage bridges.
\newblock {\em Electronic Communications in Probability}, 8:155--166, 2003.

\bibitem{doney91}
R.A. Doney.
\newblock Hitting probabilites for spectrally positive {L}\'evy processes.
\newblock {\em J. London Math. Soc.}, 44:566--576, 1991.

\bibitem{konstantopoulosetal}
T.~Konstantopoulos, A.~Kyprianou, and P.~Salminen.
\newblock On the excursions of reflected local time processes and stochastic
  fluid queues.
\newblock In {P. Glynn and T. Mikosch and T. Rolski}, editor, {\em New
  {F}rontiers in {A}pplied {P}robability - {A} {F}estschrift for {S}oeren
  {A}smussen}, Journal of Applied Probability, Spec. Vol. 48A, p. 79-98, 2011.

\bibitem{pk17}
P.~Krapivsky.
\newblock Forager on a line.
\newblock Private communication, May 2017.

\bibitem{kyprianou14}
A.E. Kyprianou.
\newblock {\em Fluctuations of L\'evy Processes with Applications. Introductory
  Lectures. 2nd ed.}
\newblock Springer-Verlag, Berlin, {H}eidelberg, 2014.

\bibitem{levy48}
P.~L{\'e}vy and M.~Lo{\`e}ve.
\newblock {\em Processus stochastiques et mouvement brownien}.
\newblock Gauthier-Villars, 1948.

\bibitem{mannersalonorros}
P.~Mannersalo, I.~Norros, and P.~Salminen.
\newblock A storage process with local time input.
\newblock {\em Queueing {S}ystems}, 46:557--577, 2004.

\bibitem{revuzyor01}
D.~Revuz and M.~Yor.
\newblock {\em Continuous Martingales and {B}rownian Motion, 3rd edition}.
\newblock Springer Verlag, Berlin, Heidelberg, 2001.

\bibitem{rogers81}
L.C.G. Rogers.
\newblock Williams' characterization of the {B}rownian excursion law : proof
  and applications.
\newblock In {J. Az\'ema} and {M. Yor}, editors, {\em S\'eminaire de
  Probabilit\'es XV}, number 850 in Springer Lecture Notes in Mathematics,
  pages 227--250, Berlin, Heidelberg, New York, 1981.

\bibitem{salminenvalloisyor07}
P.~Salminen, P.~Vallois, and M.~Yor.
\newblock On the excursion theory for linear diffusions.
\newblock {\em Japan. J. Math.}, 2:97--127, 2007.

\bibitem{shepp79}
L.A.~Shepp.
\newblock The joint density of the maximum and its location for a wiener
  process with drift.
\newblock {\em Journal of Applied probability}, pages 423--427, 1979.

\bibitem{vis11}
G.M. Viswanathan, M.G.E. da~Luz, E.P. Raposo, and H.E. Stanley.
\newblock {\em The Physics of Foraging: An Introduction to Random Searches and
  Biological Encounters}.
\newblock Cambridge University Press, 2011.

\bibitem{williams74}
D.~Williams.
\newblock Path decompositions and continuity of local time for one-dimensional
  diffusions.
\newblock {\em Proc.\ London Math.\ Soc.}, 28:738--768, 1974.

\end{thebibliography}
\end{document}